\documentclass[11pt]{article}
 \usepackage[T1]{fontenc}
\usepackage{amsthm}
\usepackage{amssymb}
\usepackage{amsmath}
\usepackage{comment}
\usepackage{thm-restate}
\usepackage{latexsym}
\usepackage{graphicx,enumerate,color}
\usepackage{subfigure}
\usepackage{url} 
\usepackage[affil-it]{authblk}
\usepackage{hyperref}
\usepackage[noabbrev,capitalise]{cleveref}

\usepackage{color}
\usepackage[normalem]{ulem}

\usepackage[margin=1in]{geometry}
\renewcommand{\v}{\textup{\textsf{v}}}

\theoremstyle{plain}
\newtheorem{thm}{Theorem}[section]
\newtheorem{lem}[thm]{Lemma}

\newtheorem{cor}[thm]{Corollary}

\newtheorem{conj}[thm]{Conjecture}

{\noindent \emph{Proof.} {}{#1}{}}{\hfill
	$\Diamond$\vspace{1em}}

\theoremstyle{plain} % just in case the style had changed
\newcommand{\thistheoremname}{}
\newtheorem{genericthm}[section]{\thistheoremname}

\theoremstyle{definition}

\newcommand{\less}{\setminus}
 
 \def\dfn#1{{\sl #1}}

 \def\se{\succcurlyeq}
\begin{document}
\date{}
 
\title{Properties of $8$-contraction-critical graphs with no $K_7$ minor}
\author{Martin Rolek\thanks{Department  of Mathematics, Kennesaw State University, Kennesaw, GA 30060, USA.  E-mail: mrolek1@kennesaw.edu.} \hskip 1cm   Zi-Xia Song\thanks{Department  of Mathematics, University of Central Florida, Orlando, FL 32816, USA. Supported by  NSF award  DMS-1854903.  E-mail:    Zixia.Song@ucf.edu.}\hskip 1cm  Robin Thomas\thanks{School of Mathematics, Georgia Institute of Technology, 
Atlanta, GA 30332, USA}}

 \date{August 22, 2022}
\maketitle
\begin{abstract}
 Motivated by the famous  Hadwiger's Conjecture,    
  we study the properties of $8$-contraction-critical graphs with no $K_7$ minor; we prove that every  $8$-contraction-critical graph with no $K_7$ minor has at most one vertex of degree $8$, where a graph $G$ is $8$-contraction-critical if $G$  is not $7$-colorable but   every proper minor   of $G$ is $7$-colorable. This is one step in our effort to prove that  every graph with no $K_7$  minor is $7$-colorable, which remains open.  
\end{abstract}

\baselineskip=16pt
\section{Introduction}

   All graphs in this paper are finite and simple.  For a graph $G$ we use $|G|$, $e(G)$, $\delta (G)$,   $\alpha(G)$, $\chi(G)$ to denote the number
of vertices, number of edges,   minimum degree,  independence number, and chromatic number  of $G$, respectively.  A graph  $H$ is a \dfn{minor} of a graph $G$ if  $H$ can be
 obtained from a subgraph of $G$ by contracting edges.  We write $G\se H$ if 
$H$ is a minor of $G$.
In those circumstances we also say that  $G$ has an  \dfn{$H$ minor}.  A graph $G$ is $k$-contraction-critical if $\chi(G)=k$ but  $\chi(H)\le k-1$  for every proper minor $H$ of $G$. Let $G$ be a graph. For a vertex $x\in V(G)$, we will use $N(x)$ to denote the set of vertices in $G$ which are adjacent to $x$.
We define $N[x] = N(x) \cup \{x\}$.  The degree of $x$ is denoted by $d_G(x)$ or
simply $d(x)$.   
The subgraph of $G$ induced by $A\subseteq V(G)$, denoted by $G[A]$, is the graph with vertex set $A$ and edge set $\{xy \in E(G) \mid x, y \in A\}$.   For an integer $k$,  a $k$-vertex in $G$ is a vertex of degree $k$, and a $k$-clique of $G$ is a set of $k$ pairwise adjacent vertices in $G$.     We define $[k]=\{1, \ldots, k\}$ for all $k\ge1$. We use $K_n, C_n, P_n$ to denote the complete graph, cycle, and path on $n$ vertices, respectively. \medskip

Our work is motivated by  the famous  Hadwiger's Conjecture~\cite{Had43}.

\begin{conj}[Hadwiger's Conjecture~\cite{Had43}]\label{HC} Every graph with no $K_t$ minor is $(t-1)$-colorable. 
\end{conj}

\cref{HC} is  trivially true for $t\le3$, and reasonably easy for $t=4$, as shown independently by Hadwiger~\cite{Had43} and Dirac~\cite{Dirac52}. However, for $t\ge5$, Hadwiger's conjecture implies the Four Color Theorem~\cite{AH77,AHK77}.   Wagner~\cite{Wagner37} proved that the case $t=5$ of Hadwiger's conjecture is, in fact, equivalent to the Four Color Theorem, and the same was shown for $t=6$ by Robertson, Seymour and  Thomas~\cite{RST}. Despite receiving considerable attention over the years, Hadwiger's Conjecture remains wide open for all $t\ge 7$,  and is    considered among the most important problems in graph theory and has motivated numerous developments in graph coloring and graph minor theory.  Proving that graphs with no $K_7$ minor are $6$-colorable is thus the first case of Hadwiger's  Conjecture that is still open.
It  is not even known yet whether  every graph with no $K_7$ minor is $7$-colorable.  Until very recently  the best known upper bound on the chromatic number of graphs with no $K_t$ minor  is $O(t (\log t)^{1/2})$,  obtained independently by Kostochka~\cite{Kostochka82,Kostochka84} and Thomason~\cite{Thomason84}, while Norin, Postle and the second  author~\cite{NPS20} showed that every graph  with no $K_t$ minor  is $O(t (\log t)^\beta)$-colorable for every $\beta>\frac14$. The current record 
 is $O(t\log \log t)$ due to
Delcourt and Postle~\cite{DelcourtPostle}.   K\"{u}hn  and Osthus~\cite{KuhOst03c} proved that Hadwiger's Conjecture is true for $C_4$-free graphs of sufficiently large chromatic number,  and for all graphs of girth at least $19$.  Kostochka~\cite{Kos14}  proved that   graphs with no $K_{s,t}$ minor  are $(s+t-1)$-colorable  for   $t>C(s\log s)^3$. We refer the reader to   a recent  paper of Lafferty and the second author~\cite{K84} on partial results towards Hadwiger's Conjecture for $t\le 9$; and  recent surveys~\cite{CV2020, K2015,Seymoursurvey} for further  background on Hadwiger's Conjecture. \medskip

The purpose of this paper  is to study the properties of $8$-contraction-critical graphs with no $K_7$ minor. This is one step in our effort to prove that  every graph with no $K_7$  minor is $7$-colorable, which remains open as mentioned above. We prove the following main result. 

\begin{thm}\label{t:main}
Let $G$ be an $8$-contraction-critical graph with no $K_7$ minor.  Let $n_i$ denote the number of $i$-vertices  in  $G$ for each $i\in\{8,9\}$.  Then 
\begin{enumerate}[(i)]
\item $8\le \delta(G)\le9$, 
\item $n_8\le 1$ and $n_9\ge 30-2n_8\ge 28$, and 
 \item for  each $9$-vertex  $v\in V(G)$,   either $G[N[v]]$ has a  $5$-clique,  or  $ \alpha(G[N(v)])=3$ and $1\le \delta(G[N(v)])\le4$.
\end{enumerate}
 
\end{thm}

Our proof of \cref{t:main} utilizes  the extremal function for $K_7$ minors (see \cref{t:exfun}),   the method for  finding   $K_7$  minors from three different  $5$-cliques    (see \cref{KNZ}), and generalized Kempe chains of contraction-critical graphs (see \cref{l:wonderful}).   
 
  \begin{thm}[Mader~\cite{exfun}]\label{t:exfun}
For each  $p\in [7]$, every graph on $n\ge p$ vertices
and at least $(p-2)n-{p-1\choose2}+1$ edges has a $K_p$ minor.
\end{thm}

A graph $G$ is said to be \dfn{apex} if there exists a vertex $v \in V(G)$ such that $G\backslash v$ is planar. The next theorem was  proved by Robertson, Seymour and Thomas~\cite{RST} to prove Hadwiger's Conjecture for $t=6$.  It is worth noting that there exists $6$-connected  apex graphs with no $K_6$ minor.

\begin{thm}\label{RST} Let $G$ be a $6$-connected non-apex graph.  If $G$  has  three  $4$-cliques, say $L_1, L_2, L_3$, such that $|L_i \cap L_j| \le 2, 1 \le i < j \le 3$, then $G \se K_6$.
\end{thm}

Theorem~\ref{RST} was then extended to $5$-cliques by Kawarabayashi and Toft~\cite{KT05} and later generalized by Kawarabayashi, Luo, Niu and Zhang~\cite{KLNZ05}.  

\begin{thm}[Kawarabayashi and Toft~\cite{KT05}]\label{KNZ}
Let $G$ be a $7$-connected graph with   $ |G|\ge 19$.  If $G$ contains three   $5$-cliques, say $L_1, L_2, L_3$, such that $|L_1 \cup L_2 \cup L_3| \ge 12$,  then  $G \se K_7$.
\end{thm}

We next list some known results on contraction-critical graphs that we shall  use later on. 
\cref{l:alpha2} below  is  a  result    of  Dirac~\cite{Dirac60} who initiated the study of   contraction-critical graphs. 
  
\begin{lem}[Dirac~\cite{Dirac60}]\label{l:alpha2}  Let $G$ be a  $k$-contraction-critical graph. Then for each  $v\in V(G)$, \[\alpha(G[N(v)])\le d(v)-k+2.\] \end{lem}

A proof of \cref{l:alpha2} can be easily obtained by contracting $v$ and a maximum  independent set of $G[N(v)]$  to a single vertex and then applying the fact that  the resulting graph is $(k-1)$-colorable.  
 Lemma~\ref{l:wonderful}   is a result of the first and second authors~\cite{RolekSong17a}, which  turns out to be very powerful because   the existence of pairwise vertex-disjoint paths  is guaranteed without using the connectivity of such  graphs.   If two vertices  $u,v$ in a graph $G$  are not adjacent,  then $uv$  is  a \dfn{missing edge} of $G$. \medskip

 \begin{lem}[Rolek and Song~\cite{RolekSong17a}]\label{l:wonderful} 
Let $G$ be any $k$-contraction-critical graph. Let $x\in V(G)$ be a vertex of
     degree $k + s$ with $\alpha(G[N(x)]) = s + 2$ and let $S \subset N(x)$ with
    $ |S| = s + 2$ be any independent set, where $k \ge 4$ and $s \ge 0$ are integers.
     Let $M$ be a set of missing edges of $G[N(x)\less S]$.  Then there
     exists a collection $\{P_{uv}:uv\in M\} $ of paths in $G$ such that
     for each $uv\in M$, $P_{uv}$ has ends $\{u, v\}$ and all its internal vertices
     in $G\less N[x]$. Moreover,  if vertices $u,v,w,z$ with $uv,wz\in M$ are distinct, then
     the paths $P_{uv}$ and $P_{wz}$ are vertex-disjoint.
 \end{lem}

 We also need a deep result of Mader~\cite{7con} on the connectivity of $8$-contraction-critical graphs. 
\begin{thm}[Mader~\cite{7con}]\label{mader} 
For all $k \ge 7$, every $k$-contraction-critical graph is $7$-connected.
\end{thm}

Finally, we shall make use of a result on rooted $K_4$ minors. 
 Let $v_1, v_2, v_3, v_4$ be four distinct vertices in a graph $G$.  We say that $G$ contains a \dfn{$K_4$  minor rooted at $v_1, v_2, v_3, v_4$} if there exist $V_1, V_2, V_3, V_4\subseteq V(G)$ such that $ v_i\in V_i$ and $G[V_i]$ is connected for each $i\in[4]$, and for $1\le i<j\le 4$, $V_i$ and $V_j$ are disjoint and there is an edge between $V_i$ and $V_j$ in $G$.  A partial answer to the next theorem was first given by Robertson, Seymour and Thomas~\cite{RST}, and later Fabila-Monroy and Wood~\cite{rootedK4}  
gave a complete characterization on graphs containing a rooted $K_4$-minor. 

\begin{thm}[\cite{rootedK4, RST}] \label{t:rootedK4}
Let $G$ be a $4$-connected graph and let  $v_1, v_2, v_3, v_4\in V(G)$ be any four distinct vertices. Then  either $G$ contains a $K_4$  minor rooted at  $v_1, v_2, v_3, v_4$, or
 $G$ is planar and $v_1, v_2, v_3, v_4$ are on a common face.
\end{thm}

 This paper is organized as follows. In the next section,  we prove \cref{t:main}(i, ii). In Section~\ref{s:deg9}, we prove \cref{t:main}(iii).\medskip 
 
 We need to introduce   more notation.  Let $G$ be a graph.  The \dfn{complement} of $G$ is denoted by $\overline{G}$. 
 If $x,y$ are adjacent
vertices of a graph $G$, then we denote by $G/xy$ (or simply $G/e$ if $e=xy$) the graph obtained from $G$
by contracting the edge $xy$ and deleting all resulting parallel
edges.   If $u,v$ are distinct nonadjacent vertices of a graph $G$, then by
$G+uv$ we denote the graph obtained from $G$ by adding an edge
with ends $u$ and $v$.  If $u,v$ are adjacent or equal, then we define
$G+uv$ to be $G$.  Similarly, if  $M \subseteq E(G)\cup E(\overline{G})$, then   by
$G+M$ we denote the graph obtained from $G$ by adding  all the edges of $M$ to $G$.  If  $A, B\subseteq V(G)$ are disjoint, we say that $A$ is \emph{complete} to $B$ if each vertex in $A$ is adjacent to all vertices in $B$, and $A$ is \dfn{anti-complete} to $B$ if no vertex in $A$ is adjacent to any vertex in $B$.
If $A=\{a\}$, we simply say $a$ is complete to $B$ or $a$ is anti-complete to $B$.  
  We denote by $B \less A$ the set $B - A$,   and $G \less A$ the subgraph of $G$ induced on $V(G) \less A$, respectively. 
If $A = \{a\}$, we simply write $B \less a$    and $G \less a$, respectively.  An $(A, B)$-path in $G$ is a path $P$ with one end in $A$ and  the other in $B$  such that  no internal vertex of $P$ belongs to $A\cup B$; we simply say $(a, B)$-path  if $A=\{a\}$.  We use $e(A, B)$ to denote the number of edges in $G$ with one end in $A$ and the other in $B$. 
   We say  that $G$ is \emph{$H$-free} for some graph $H$ if it has no subgraph isomorphic to $H$.    The \dfn{join} $G+H$ (resp. \dfn{union} $G\cup H$) of two 
vertex-disjoint graphs
$G$ and $H$ is the graph having vertex set $V(G)\cup V(H)$  and edge set $E(G)
\cup E(H)\cup \{xy\, |\,  x\in V(G),  y\in V(H)\}$ (resp. $E(G)\cup E(H)$).  
We use the convention   ``A :="  to mean that $A$ is defined to be
the right-hand side of the relation.   \medskip

 \section{Number of $8$-vertices}\label{s:deg8}

We begin this section with a lemma. 
\begin{figure}[htbp]
\centering
\includegraphics[scale=0.4]{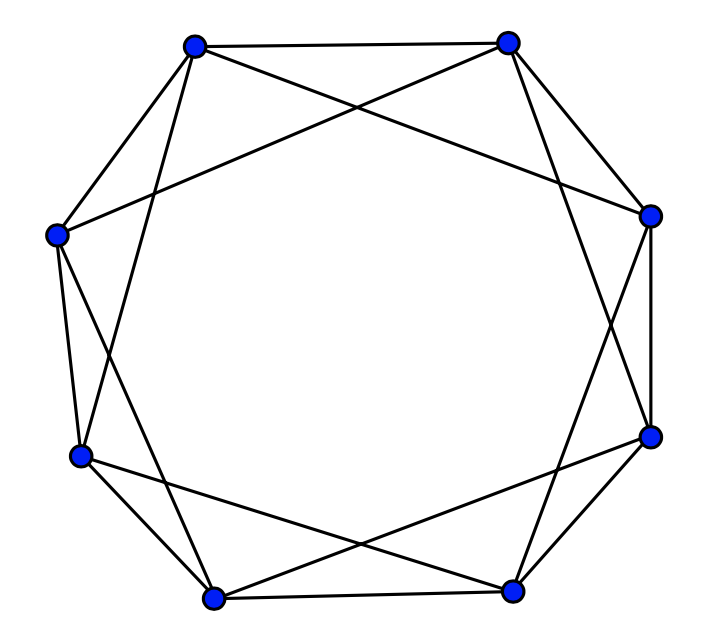}
\caption{The graph $H_8$.}
\label{H8}
\end{figure}

\begin{lem}\label{l:H8}  Let $H$ be a graph with $|H|=8$ and $\alpha(H)=2$. Then $H$  contains   $K_4$  or $H_8$ as a subgraph, where $H_8$  is depicted in Figure~\ref{H8}. 
\end{lem}
\begin{proof}   Suppose $H$ is $K_4$-free.  We   show that $H$ contains $H_8$ as a subgraph. We may assume that $H$ is edge-minimal subject to being $K_4$-free and $\alpha(H)=2$.   Let $u\in V(H)$.   Since $\alpha(H)=2$, we see that   $V(H)\less N[u]$ is a clique. Thus  $|H\backslash N[u]|\le3$ because $H$ is $K_4$-free.  Hence $d(u) \ge 4$.  On the other hand, since $\alpha(H)=2$  and $G[N(u)]$ is $K_3$-free, we see that   $d(u)\le5$ because the Ramsey number $R(3,3)=6$.   Thus\medskip

\noindent ($*$)  for each $u\in V(H)$,  $4\le d(u)\le5$ and  $G[N(u)]$ is  $K_3$-free and   $\overline{K_3}$-free.\medskip

 We next prove that  $H$ is $4$-regular. Suppose not. By ($*$), let   $x\in V(H)$ be a $5$-vertex in  $H$;   so    $G[N(x)]$ is isomorphic to  $C_5$. Let $\{y_1, y_2\}:=V(H)\less N[x]$.  Then $y_1y_2\in E(H)$; in addition,  $y_1$ and  $y_2$ have a common neighbor, say $w$, in  $N(x)$.  Then $d(w)=5$.  It follows  that   $H\less  xw$ is $K_4$-free and $\alpha(H\less xw)=2$, contrary to the minimality of $e(H)$.    Thus    $H$ is $4$-regular.   Then $\overline{H}$ is $3$-regular on $8$ vertices. Note  that $\overline{H}$  is $K_3$-free and $\alpha(\overline{H})=3$.   Let $w\in V(\overline{H})$ and   $N_{\overline{H}}(w):= \{w_1, w_2, w_3\}$.  Let $X  := \{x_1, x_2, x_3, x_4\}$ be the remaining vertices of $\overline{H}$.  Since $\overline{H}$ is $K_3$-free, we see that $N_{\overline{H}}(w)$ is an independent  set. This implies that $e(\{w_1, w_2, w_3\}, X)=6$ and so $e(\overline{H}[X])=3$ and $\alpha(\overline{H}[X])=2$. It follows that   $\overline{H}[X]=P_4$.  We may assume that $x_1, x_2, x_3, x_4$ are the vertices of $\overline{H}[X]$ in order. Then each of $x_1$ and $x_4$ has two neighbors in $N_{\overline{H}}(w)$. We may assume that $w_2$ is a common neighbor of $x_1$ and $x_4$. Since $\overline{H}$ is triangle-free, by symmetry, we may assume that $w_1$ is adjacent to $x_1$ and  $x_3$. Then $w_3$ must be adjacent to $x_2$ and $x_4$. One can easily check that $\overline{H}$ is isomorphic to  $\overline{H_8}$, and so $H$ is isomorphic to $H_8$, as desired.     \end{proof}

\begin{lem}\label{l:deg8} Let  $G$ be  an $8$-contraction-critical  graph with no $K_7$ minor. Then the following hold. 
\begin{enumerate}[(a)]
\item $8\le \delta(G)\le 9$. 
\item  $2n_8+n_9\ge30$. 
\item  For every subgraph $H$ of $G$ with $|H|\le7$, $H$ has no  $K_6$ minor.
\item  For every $8$-vertex  $v\in V(G)$,   $G[N(v)]$ has  two  disjoint $4$-cliques.  
\end{enumerate}
\end{lem} 

\begin{proof}  Since $G$ has no  $K_7$ minor, by \cref{t:exfun}, $e(G)\le 5|G|-15$ and so $\delta(G)\le 9$. On the other hand, since $G$ is $8$-contraction-critical, we have $\delta(G)\ge7$. Thus $7\le \delta(G)\le9$. Suppose  $\delta(G)=7$.  Let $x\in V(G)$ be a $7$-vertex  in $G$. By \cref{l:alpha2}, $G[N(x)]=K_7$, a contradiction.   This proves (a).  It is simple to check that  $2n_8+n_9\ge30$ because $8n_8+9n_9+10(|G|-n_8-n_9)\le 2e(G)\le10|G|-30$.
To prove (c), suppose    $G$ contains  a subgraph $H$ such that  $|H|\le7$ and  $H\se K_6$. Let  $x \in V(G) \less V(H)$.   By Theorem~\ref{mader}, $G$ is $7$-connected and so by Menger's theorem,  there exist $|H|$ internally disjoint $(x, V(H))$-paths, say   $Q_1, \dots, Q_{|H|}$.   By contracting all the edges of   $Q_1\less x,\ldots,  Q_{|H|}\less x$,  we obtain a $K_7$  minor of $G$, a contradiction. \medskip

 \begin{figure}[htbp]
\centering
\includegraphics[scale=0.5]{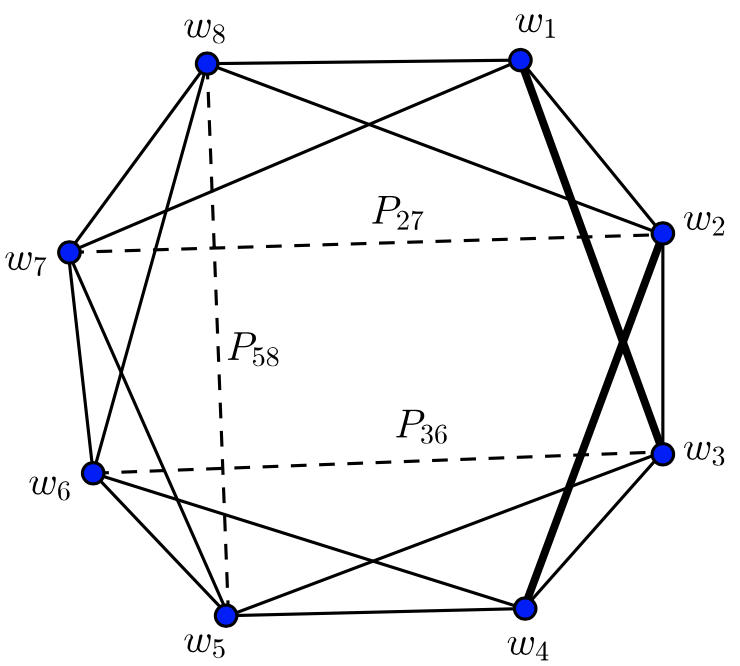}
\caption{\small $H_8$ with $P_{27}$, $P_{36}$ and $P_{58}$  shown as dotted lines, edges $w_1w_3, w_2w_4$ in bold lines. }
\label{H8b}
\end{figure}

It remains to prove (d). 
Let $v\in V(G)$ be an $8$-vertex.  By \cref{l:alpha2}, $\alpha(G[N(v)])\le2$. Since $G$ has no $K_7$ minor, we see that $\alpha(G[N(v)])=2$.  We claim  that  \medskip

\noindent ($*$) $G[N(v)]$  has a $4$-clique.    \medskip

 Suppose    $G[N(v)]$  is $K_4$-free.   By Lemma~\ref{l:H8}, $G[N(v)]$ contains $H_8$ as a subgraph.   Let $ w_1, \dots, w_8 $ be the vertices of $H_8$, as depicted in  Figure \ref{H8b}.  Then  $w_1w_4\notin E(G)$ because $G[N(v)]$ is $K_4$-free. We may assume that $M=\{w_2w_7, w_3w_6, w_5w_8\}$ is a set of missing edges of $G[N(v)]$. 
By Lemma~\ref{l:wonderful} applied to $N(v)$ with $S=\{w_1, w_4\}$ and $M=\{w_2w_7, w_3w_6, w_5w_8\}$, 
there exist pairwise vertex-disjoint paths $P_{27}$ with ends $w_2, w_7$, $P_{36}$ with ends $w_3, w_6$, $P_{58}$ with ends $w_5, w_8$, and all their internal vertices in $G\less N[v]$. Now by contracting each of the edges $w_1w_3$, $w_2w_4$ to a single vertex, and then all the edges  of  $P_{27}\less w_2$, $P_{36}\less w_3$ and $P_{58}\less w_5$,  we see that  $G\se   K_7$, a contradiction. This proves ($*$).\medskip

By  ($*$),  let  $W:=\{w_1, w_2, w_3, w_4\}\subseteq  N(v)$  be  a $4$-clique in $G$.  Let $w_5, w_6, w_7, w_8$ be the   vertices of $N(v)\less W$.  Let $J := G[\{w_5, w_6, w_7, w_8\}]$.  Suppose   $J\ne K_4$.   By Lemma~\ref{l:deg8}(c), $G[N(v)]$ is $K_5$-free.  Thus  for each $j\in\{5,6,7,8\}$,   $w_j $  has at least one non-neighbor in $W$;    by Lemma~\ref{l:deg8}(a),  $w_j$  is adjacent to at least one vertex in $G\backslash N[v]$.  We next prove  that $J$ is isomorphic to $K_3\cup K_1$.   \medskip

Suppose   $J$ contains   $P_3$  as an induced subgraph. We may assume that $w_5, w_6, w_7$ are the vertices of $P_3$ in order. Since $\alpha(G[N(v)])=2$, we see that  $w_j$ is adjacent to $w_5$ or $w_7$ for all  $j\in \{1,2,3,4,8\}$.   We may assume that $w_8$ is anti-complete  to $\{w_1, \dots, w_t\}$ for some   $t\in[4]$.  By Lemma~\ref{l:wonderful} applied to $N(v)$ with $S=\{w_5, w_7\}$ and $M=\{w_8w_1, \dots, w_8w_t\}$, there exist  $t$ pairwise internally vertex-disjoint paths $Q_{81},   \dots, Q_{8t}$, where each $Q_{8j}$ has ends $w_8, w_j$ and all its internal vertices in $G\less N[v]$ for all $j\in\{1,2, \dots, t\}$. Now by contracting $P_3$ to a single vertex,  and then all the edges  of  $Q_{81}\less w_1,   \dots, Q_{8t}\less w_t$ onto $w_8$,  we see that  $G\se K_7$, a contradiction.
This proves that $J$ does not contain  $P_3$ as an induced subgraph.   Suppose next that  $J$ is isomorphic to  $K_2\cup K_2$. 
We may assume that  $w_5 w_6$ and $w_7 w_8$ are the two edges of $J$.   By Theorem~\ref{mader},  $G \backslash \{v, w_1, w_2, w_3, w_4\}$ is 2-connected. Thus $G \backslash \{v, w_1, \dots, w_4\}$ contains  two vertex-disjoint paths, say $Q_1$ and $Q_2$,  between  $\{w_5, w_6\}$ and $\{w_7, w_8\}$.  We may assume that $Q_1$ has ends $w_5, w_7$ and $Q_2$ has ends $w_6, w_8$.  Since $\{w_5, w_7\}$ and $\{w_6, w_8\}$ are independent sets of size $2$, every vertex in $\{w_1, \dots, w_4\}$ must be adjacent to at least one vertex in  $\{w_5, w_7\}$ and $\{w_6, w_8\}$, respectively.  By contracting $Q_1$ and $Q_2$ to two distinct  vertices, together with $x, w_1, \ldots, w_4$, we see that  $G\se K_7$,   a contradiction. This proves that  $J$ is not isomorphic to  $K_2\cup K_2$.   It follows that  $J$ is isomorphic to $K_3\cup K_1$, because $\alpha(J)=2$ and $J$ does not contain $P_3$ as an induced subgraph.  \medskip

We may assume that $d_J(w_8)=0$.  Then  $J[\{w_5, w_6, w_7\}]=K_3$.  Since $G[N(v)]$ is $K_5$-free and $\alpha(G[N(v)])=2$, we see that $w_8$ has   exactly one non-neighbor, say $w_1$, in $W$. Then $w_1$ is complete to $\{w_5, w_6, w_7\}$ and $w_8$ is complete to $\{w_2, w_3, w_4\}$. Therefore  $G[N(v)]$ has two disjoint $4$-cliques $\{w_8, w_2, w_3, w_4\}$ and $\{w_1, w_5, w_6, w_7\}$, as desired. \end{proof}  

 Lemma~\ref{l:deg8}(d) implies  the following:

\begin{cor}\label{2K5}
 Let $G$ be an  $8$-contraction-critical graph with no $K_7$ minor. Then   every $8$-vertex $v$ in  $G$ belongs to two  $5$-cliques  having only $v$ in common. 
\end{cor}

\begin{figure}[htbp]
\centering
\subfigure[][\label{threeK5a}]{
\hfill\includegraphics[scale=0.45]{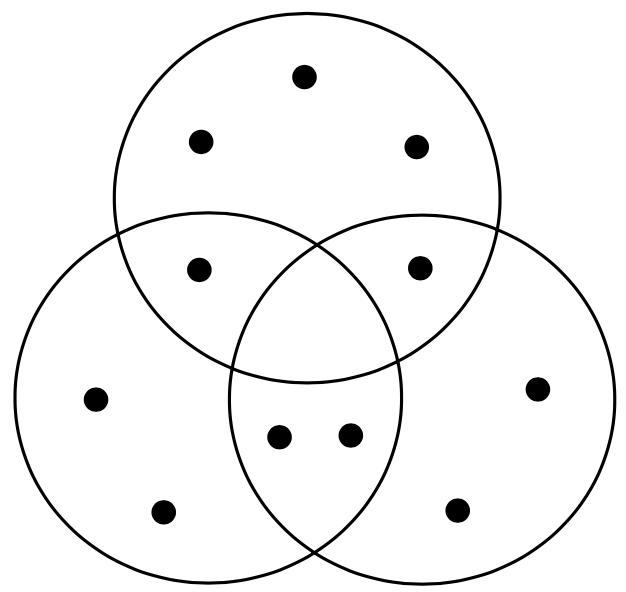}
}
\hskip 2cm
\subfigure[][\label{threeK5b}]{
\hfill\includegraphics[scale=0.6]{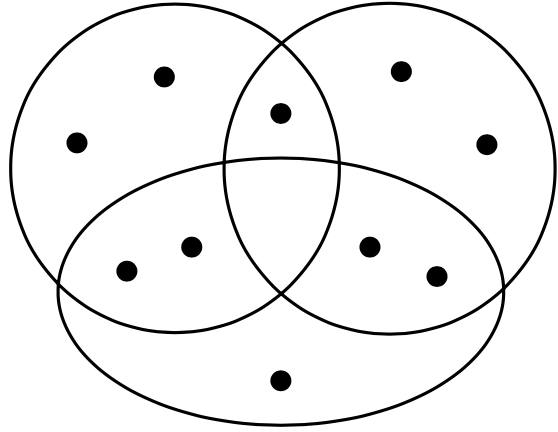}
}
\caption{Two different cases of three $5$-cliques.}
\label{threeK5}
\end{figure}

\begin{lem}\label{3K5}
 Let $G$ be an $8$-contraction-critical graph. If $G$ has     two different $5$-cliques   with exactly three vertices in common or three different $5$-cliques as depicted in Figure~\ref{threeK5}, then  $G \se K_7$.
\end{lem}

 \begin{proof}      Assume first that   $G$ has  two different $5$-cliques $L_1$ and $L_2$ such that   $L_1 \cap L_2 = \{w_1, w_2, w_3\}$.  Let $v_1, v_2$ be the remaining vertices of $L_1$ and $v_3, v_4$ the remaining vertices of $L_2$.  By Theorem~\ref{mader}, $H :=G\less \{w_1, w_2, w_3\}$ is 4-connected. Then $H$ must be non-planar,  otherwise   $\chi(G)\le 7$, a contradiction.  By Theorem~\ref{t:rootedK4}, $H$ contains a $K_4$  minor   
 rooted at $v_1, \dots, v_4$ and so $G\se K_7$.   \medskip
 
 Assume next  that  $G$ has  three different $5$-cliques  $L_1$, $L_2$  and $L_3$ as given in Figure~\ref{threeK5}.  We  first consider the case that $L_1$, $L_2$  and $L_3$ are as depicted in Figure~\ref{threeK5}(a).   Let $L_1 = \{v_1, v_2, w_1, x_1, x_2\}$, $L_2 = \{v_1, v_2, w_2, x_3, x_4\}$ and $L_3 = \{w_1, w_2, y_1, y_2, y_3\}$.   By Theorem~\ref{mader}, $G\less\{v_1, v_2, w_1, w_2\}$ is 3-connected.  By Menger's theorem,  there exist three pairwise vertex-disjoint paths $Q_1, Q_2, Q_3$ between  $\{y_1, y_2, y_3\}$ and $\{x_1, x_2, x_3,  x_4\}$ in $G \less  \{v_1, v_2, w_1, w_2\}$.  By contracting each of $Q_1$, $Q_2$ and $Q_3$  to a single vertex, together with $v_1, v_2, w_1, w_2$, we see that  $G\se K_7$.  Finally we consider the case   that $L_1$, $L_2$ and $L_3$ are as depicted in Figure~\ref{threeK5}(b).  Let $L_1 = \{v_1, v_2, w_1, w_2, u\}$, $L_2 = \{v_1, v_2, x_1, x_2, y\}$ and $L_3 = \{w_1, w_2, z_1, z_2, y\}$.  By Theorem~\ref{mader},  $G \less \{v_1, v_2, w_1, w_2, y\}$ is 2-connected. By Menger's theorem,  there exist two vertex-disjoint paths $R_1, R_2$ between  $\{x_1, x_2\}$ and $\{z_1, z_2\}$ in $G \less \{v_1, v_2, w_1, w_2, y\}$.  Now contracting each of  $R_1$ and $R_2$ to a single vertex, together with $v_1, v_2, w_1, w_2, y$,  yields a $K_7$ minor in $G$.  \end{proof}

  We are now ready to prove the main result of this section.

\begin{thm}\label{t:n8}
 Let  $G$ be an $8$-contraction-critical  graph with no $K_7$ minor. Then   $n_8\le1$.  
 \end{thm}
 
 \begin{proof} Suppose    $n_8\ge2$. Let  $u, v \in V(G)$ be two  distinct vertices of degree 8 in $G$.  By Corollary~\ref{2K5},  let $L_1$ and $L_2$ be two $5$-cliques of  $G[N[u]]$ with $L_1\cap L_2=\{u\}$, and  $L_3$ and $L_4$ be two $5$-cliques of  $G[N[v]]$ with $L_3\cap L_4=\{v\}$. For each $i\in\{1,2\}$ and each $j\in\{3,4\}$, by  Theorem~\ref{KNZ}, $|L_1\cup L_2\cup L_j| \le 11$ and $|L_3\cup L_4\cup L_i| \le 11$. It follows that    $3 \le |L_j \cap (L_1 \cup L_2)| \le 5$  and $3 \le |L_i \cap (L_3 \cup L_4)| \le 5$.    By Lemma~\ref{3K5}, $|L_i \cap L_j| \ne 3$ for each  $ i\in\{1,2\}$ and each $ j\in\{3, 4\}$. \medskip
 
 We  claim that $uv\in E(G)$. Suppose    $uv\notin E(G)$. Then $u\notin L_3$ 
 and  so $3 \le |L_3 \cap (L_1 \cup L_2)| \le 4$.  Suppose  that $|L_3 \cap (L_1 \cup L_2)| = 4$. Let $w \in L_4 \cap (L_1 \cup L_2)$.  Note that $w \ne u, v$ and $w \notin L_3$.  Then  $G[L_3 \cup \{u, w\}]/uw= K_6$, contradicting Lemma~\ref{l:deg8}(c).  Thus $|L_3 \cap (L_1 \cup L_2)| = 3$.  Since  $|L_3 \cap L_1| \ne 3$ and $|L_3 \cap L_2| \ne 3$,   we may assume that $|L_3 \cap L_1| = 2$ and $|L_3 \cap L_2| = 1$.  But then $L_1$,   $L_2$ and $L_3$ are as depicted in Figure~\ref{threeK5}(a),  by Lemma~\ref{3K5}, $G \se K_7$, a contradiction. Thus $uv\in E(G)$, as claimed. \medskip
 
  We may assume  that $v \in L_1$ and $u \in L_3$.  Then $u, v\in L_1\cap L_3$,  and so $v \notin L_2$ and $u \notin L_4$.  By Lemma~\ref{l:deg8}(c),    $G[L_2\cup\{v\}]$ is $K_6$-free. Thus there exists $w\in L_2$ such that $vw\notin E(G)$. Similarly, there exists $z\in L_4$ such that $uz\notin E(G)$.  Then $u, w\notin L_4$   and $v, z\notin L_2$.  Thus   $|L_2 \cap L_4|\le3$.  Since $|L_2 \cap L_4|\ne3$,   we have  $|L_2 \cap L_4|\le2$.  Suppose $1\le |L_2 \cap L_4|\le2$. Let $z^*\in L_2\cap L_4$. Then $z, z^*\notin L_1$ and so $ |L_1 \cap L_4|\le3$. Recall that   $|L_4\cap(L_1\cup L_2)| \ge3$. Thus  $ |L_1 \cap L_4|\ge1$.  Since  $|L_1\cap L_4|\ne3$, we see that       $1\le |L_1 \cap L_4|\le2$.  Now it is straightforward to  check that   $L_1$, $L_2$ and $L_4$ are as depicted in Figure~\ref{threeK5}(a) if $|L_1\cap L_4|=1$ and $|L_2\cap L_4|=2$, or $|L_1\cap L_4|=2$ and $|L_2\cap L_4|=1$; and in Figure~\ref{threeK5}(b) if $|L_1\cap L_4|=2$ and $|L_2\cap L_4|=2$.  By Lemma~\ref{3K5}, $G\se K_7$, a contradiction.   This proves that  $L_2 \cap L_4=\emptyset$.   Since $|L_1\cap L_4|\ne3$  and $|L_1\cup L_2\cup L_4| \le 11$, we must have $L_4\less L_1=\{z\}$ (i.e., $|L_1\cap L_4|=4$). Similarly, $L_2\less L_3=\{w\}$ (i.e., $|L_2\cap L_3|=4$).
  Now $L_1, L_2, L_3, L_4$ are as depicted in Figure~\ref{2Deg8}.   \medskip
  
 \begin{figure}[htbp]
\centering
\includegraphics[scale=0.3]{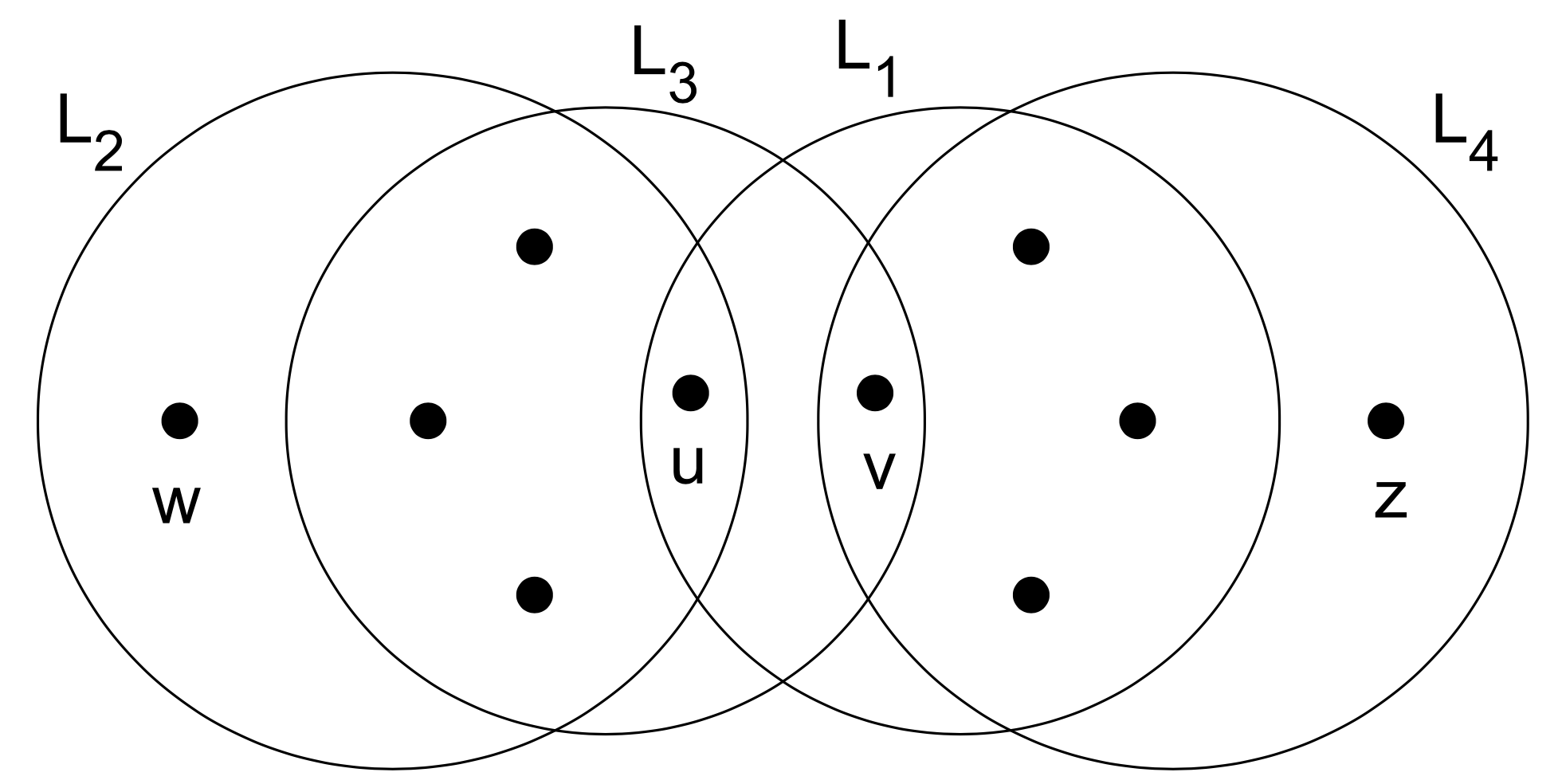}

\caption{ Two  $8$-vertices $u, v$. }
\label{2Deg8}
\end{figure}

Let $L_1\less\{u,v\}:=\{y_1, y_2, y_3\}$ and $L_3\less \{u,v\}:=\{x_1, x_2, x_3\}$.  Since  $G$ is $7$-connected,  there exist four pairwise vertex-disjoint paths, say $Q_1, Q_2, Q_3, Q_4$, between  $\{w, x_1, x_2, x_3\}$ and $\{z, y_1, y_2, y_3\}$ in  $G\less\{u,v\}$. We may assume that $Q_4$  has  ends $x_3$ and $y_j$ for some $j\in[3]$.  By contracting each of $Q_1, Q_2, Q_3, Q_4\less x_3$ to a single vertex, together with $u, v, x_3$, we see that  $G\se K_7$, a contradiction. \medskip

  This completes the proof of \cref{t:n8}.  \end{proof}

   \section{$9$-vertices}\label{s:deg9}
 
 By \cref{l:deg8}(b) and \cref{t:n8}, we see that every $8$-contraction-critical graph with no $K_7$ minor has at least $28$ $9$-vertices. In this section, we study the properties of $G[N(x)]$ for such     $9$-vertices $x$.    
  
 \begin{lem}\label{l:deg9}
Let $H$ be a graph with  $|H| = 9$ and $\delta(H) \ge 5$.  
If $H$ is $K_4$-free, then  either $H\se K_6$, or $H$ is isomorphic to $\overline{K_3}+C_6$. 
\end{lem}

\begin{proof}  \cref{l:deg9}  can be checked by computers, see Appendix. We have a  computer-free proof for  \cref{l:deg9} but  is long, we omit it here.\end{proof}

 \begin{lem}\label{K5}

 Let $G$ be an $8$-contraction-critical graph with no  $K_7$ minor and let   $v$ be a   $9$-vertex   in $G$. Then 
 \begin{enumerate}[(i)]
 
 \item $G[N[v]]$  has a $4$-clique, and 
 \item   either $G[N[v]]$ has a  $5$-clique,  or  $1\le \delta(G[N(v)])\le4$ and $\alpha(G[N(v)])=3$.
\end{enumerate}
\end{lem}

\begin{proof}   By \cref{l:alpha2},   $\alpha(G[N(v)])\le3$. 
 To prove (i),  suppose $G[N(v)]$ is $K_3$-free. Let $u\in N(v)$.    Since   $\alpha(G[N(v)])\le3$, we see that  $G[N(v)\cap N[u]]$   has a $3$-clique   if $|N(v)\cap N(u)|\ge4$,  a contradiction.  Thus  $|N(v)\cap N(u)|\le3$. If $|N(v)\cap N(u)|\le2$, then  $|N(v)\less N[u]|\ge6$ and $\alpha(G[N(v)\less N[u]])\le2$; thus    $G[N(v)\less N[u]$  has a   $3$-clique because the Ramsey number $R(3,3)=6$, a contradition.   Thus $|N(v)\cap N(u)|=3$ and so  $d_{G[N(v)]}(u)=3$. By the arbitrary  choice of $u$,  we see that $G[N(v)]$ is 3-regular, which is impossible because $|N(v)|=9$.  This proves (i). \medskip

To prove (ii),  suppose  $G[N(v)]$ is $K_4$-free.  We next show that $1\le \delta(G[N(v)])\le4$   and  $\alpha(G[N(v)])=3$.   Suppose   $\delta(G[N(v)])\ge5$.  Since  $G$  has no   $K_7$ minor, by \cref{l:deg9},  $G[N(v)]$ is isomorphic to $\overline{K_3}+C_6$. Let $x, y$ be any two non-adjacent vertices of $C_6$.  Since $G$ is $7$-connected, there must exist $x,y\in N(v)$ such that $x, y$ are non-adjacent vertices of $C_6$ and  there exists  an $(x,y)$-path $P$ with internal vertices in $G\less N[x]$. By contracting   $P\less x$ onto $y$, we see that $G\se G[N[v]]+xy\se K_7$, a contradiction.    This proves that $\delta(G[N(v)])\le4$. 
Next suppose $\delta(G[N(v)])=0$.  Let $u\in N(v)$ be an isolated vertex in $G[N(v)]$. Since   $\alpha(G[N(v)])\le3$ and $G$ has no $K_7$ minor, we see that $\alpha(G[N(v)]\less u)=2$.  Then $G[N(v)]\less u$ is a graph on $8$ vertices with $\alpha(G[N(v)]\less u)=2$. By Lemma~\ref{l:H8}, $G[N(v)]\less u$ contains  $H_8$ as a subgraph.  Similar to the proof of \cref{l:deg8}(d), let $w_1, \dots, w_8$ be the vertices of $H_8$, as depicted in  Figure \ref{H8b}.  Then $w_1w_4\notin E(G)$ because $G[N(v)]$ is $K_4$-free. By Lemma~\ref{l:wonderful} applied to $N(v)$ with $S=\{w_1, w_4, u\}$ and $M=\{w_2w_7, w_3w_6, w_5w_8\}$, 
there exist pairwise vertex-disjoint paths $P_{27}$ with ends $w_2, w_7$, $P_{36}$ with ends $w_3, w_6$, $P_{58}$ with ends $w_5, w_8$, and all their internal vertices in $G\less N[v]$. Now by contracting each of the edges $w_1w_3$, $w_2w_4$ to a single vertex, and then all the edges  of  $P_{27}\less w_2$, $P_{36}\less w_3$ and $P_{58}\less w_5$ onto $N[v]$,  we see that  $G\se G[N[v]]+M\se  K_7$, a contradiction.  This proves that  $1\le \delta(G[N(v)])\le4$. \medskip

 It remains to show that  $\alpha(G[N(v)])=3$.  Suppose $\alpha(G[N(v)])=2$. Let $x \in N(v)$ such that   $d(x)=\delta(G[N(v)])$.  Let $A:=N[x]\cap N(v)$. Then  $|A|\le 5$ because $\delta(G[N(v)])\le4$. But then  $N(v)\less A$ is a clique of order $9-|A|\ge4$ because   $\alpha(G[N(v)])=2$, a contradiction.  
  \end{proof}

%
%\bibliographystyle{alpha}
%\bibliography{Zixia}

\section*{Appendix}
The codes we use for \cref{l:deg9} are provided on the next three pages (we follow the small program  available on  the third author's  website \href{https://thomas.math.gatech.edu/PAP/K9/}{https://thomas.math.gatech.edu/PAP/K9/}). Our small  program found five graphs $H$ with $|H|=9$,  $\delta(H)\ge5$ and no $K_6$ minor, and subject to these, $e(H)$ is minimal, that is, every edge in $H$ is incident with a $5$-vertex: $K_1+H_8$ (where $H_8$ is given in Figure~\ref{H8}), $\overline{K_3}+C_6$,  and three  more  given in Figure~\ref{f:NoK6}; each of these five graphs, except $\overline{K_3}+C_6$, has a $4$-clique.  

\begin{figure}[h]
\centering
\subfigure[][\label{fig:Gamma1}]{\includegraphics[scale=0.25]{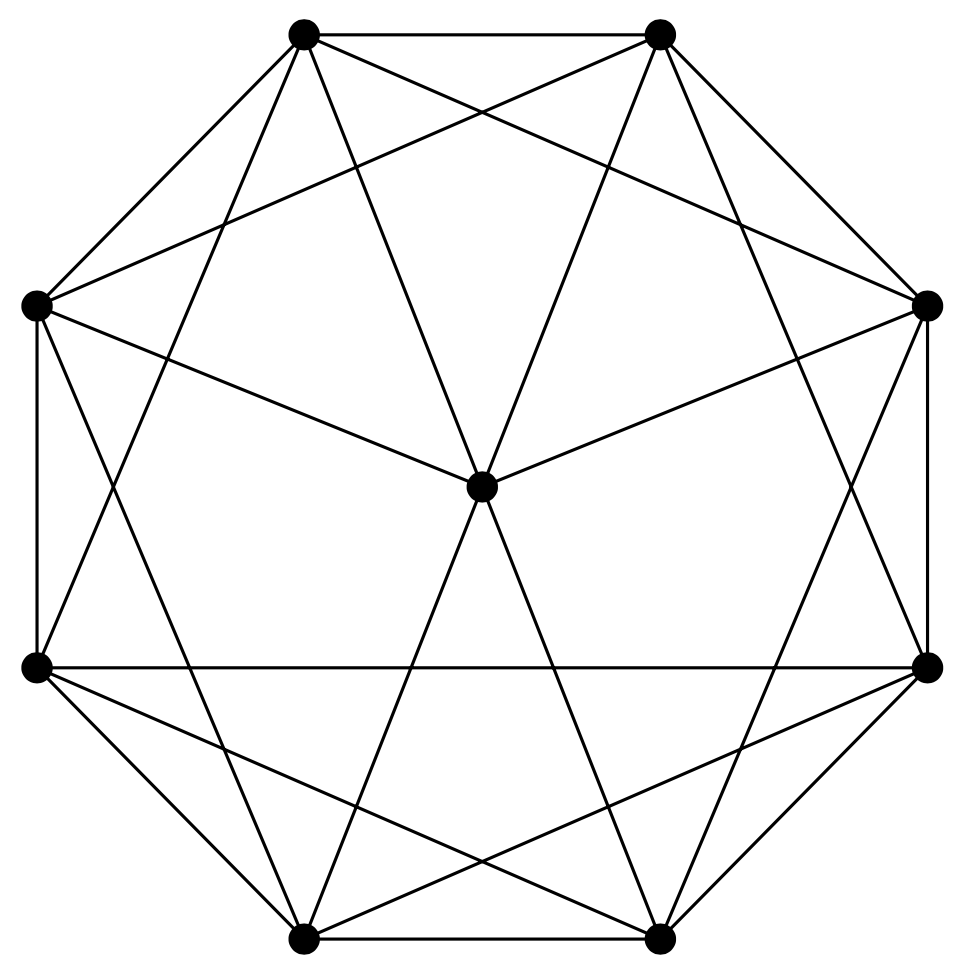}}
\quad\quad\quad
\subfigure[][\label{fig:Gamma2}]{\includegraphics[scale=0.26]{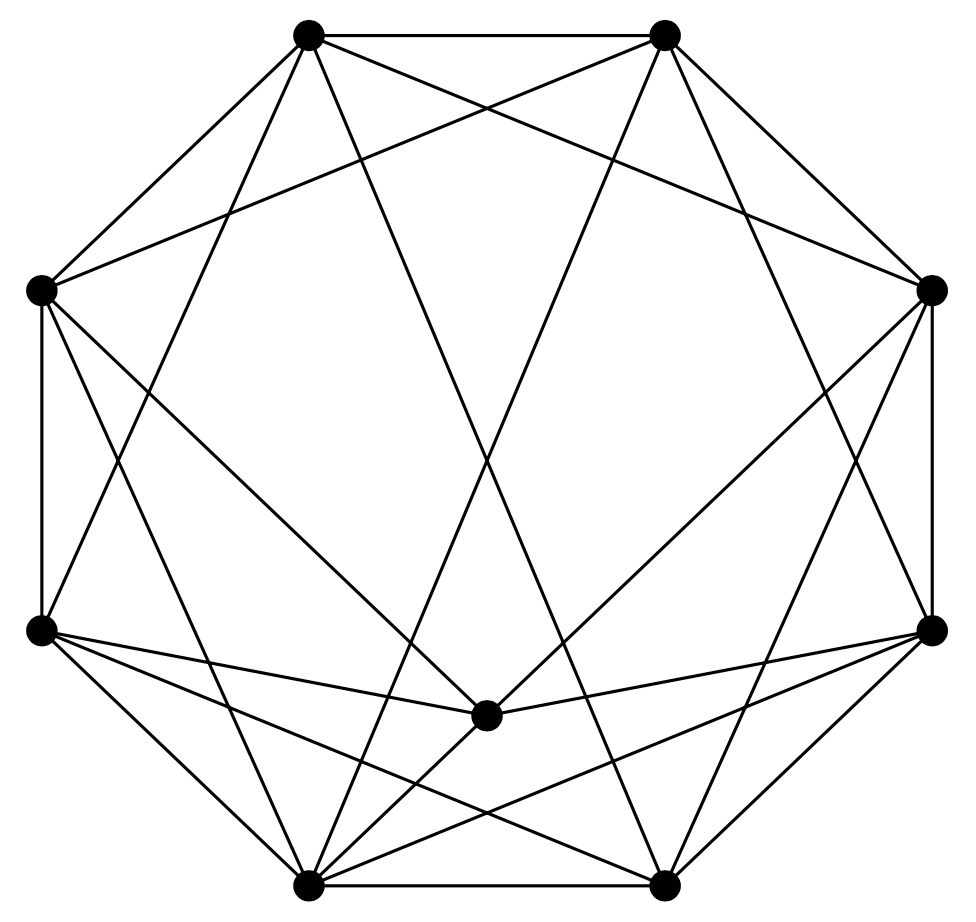}}
\quad\quad\quad
\subfigure[][\label{fig:Gamma3}]{\includegraphics[scale=0.185]{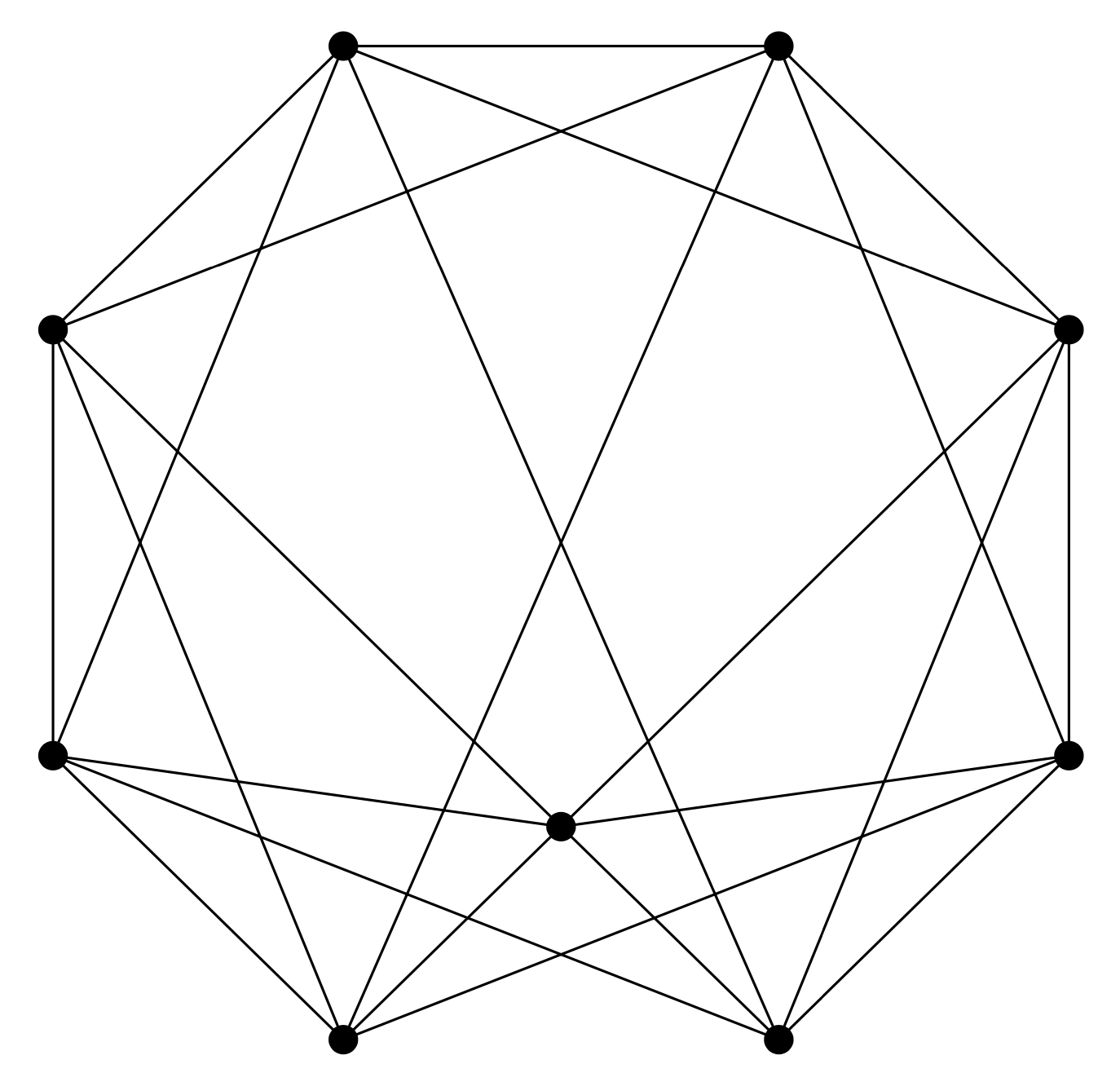}}
\caption{Three graphs with no $K_6$ minor.}
\label{f:NoK6}
\end{figure}
 
\begin{figure}[h]
\centering
 \includegraphics[scale=0.8]{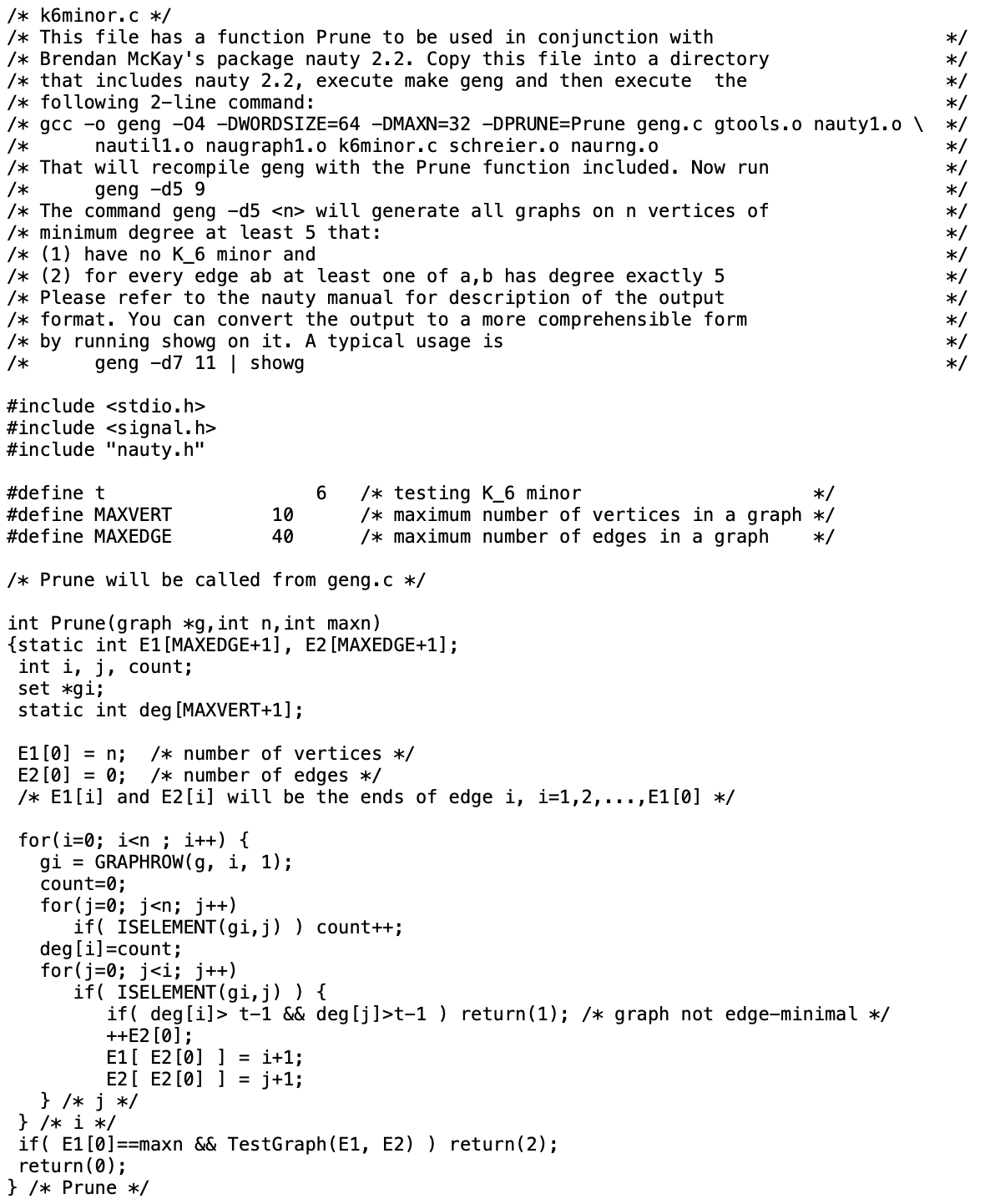}
 \end{figure}
\hskip 5cm

\begin{figure}[h]
\centering
 \includegraphics[scale=0.85]{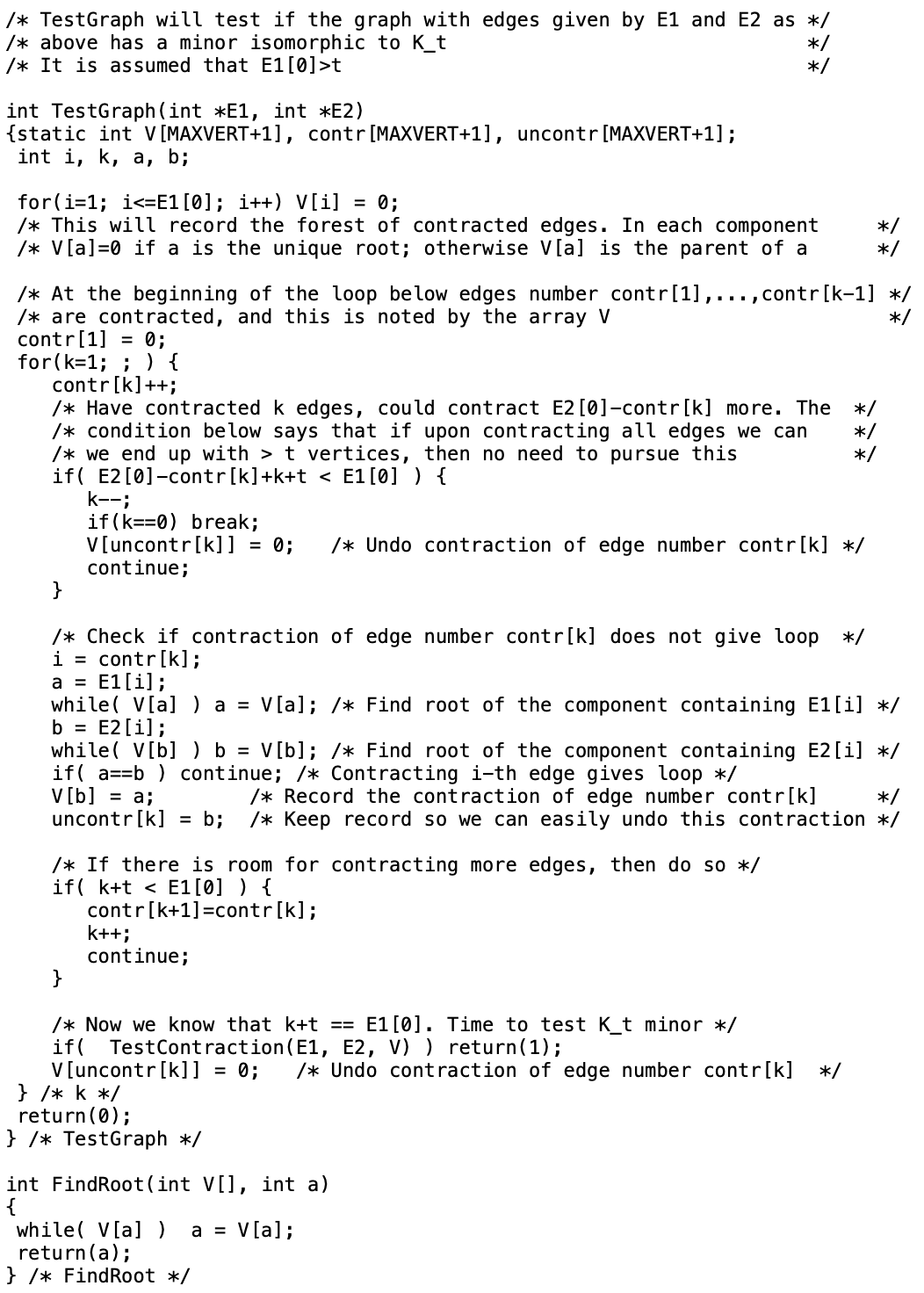}
 \end{figure}

\hskip 5cm
\begin{figure}[h]
\centering
 \includegraphics[scale=0.85]{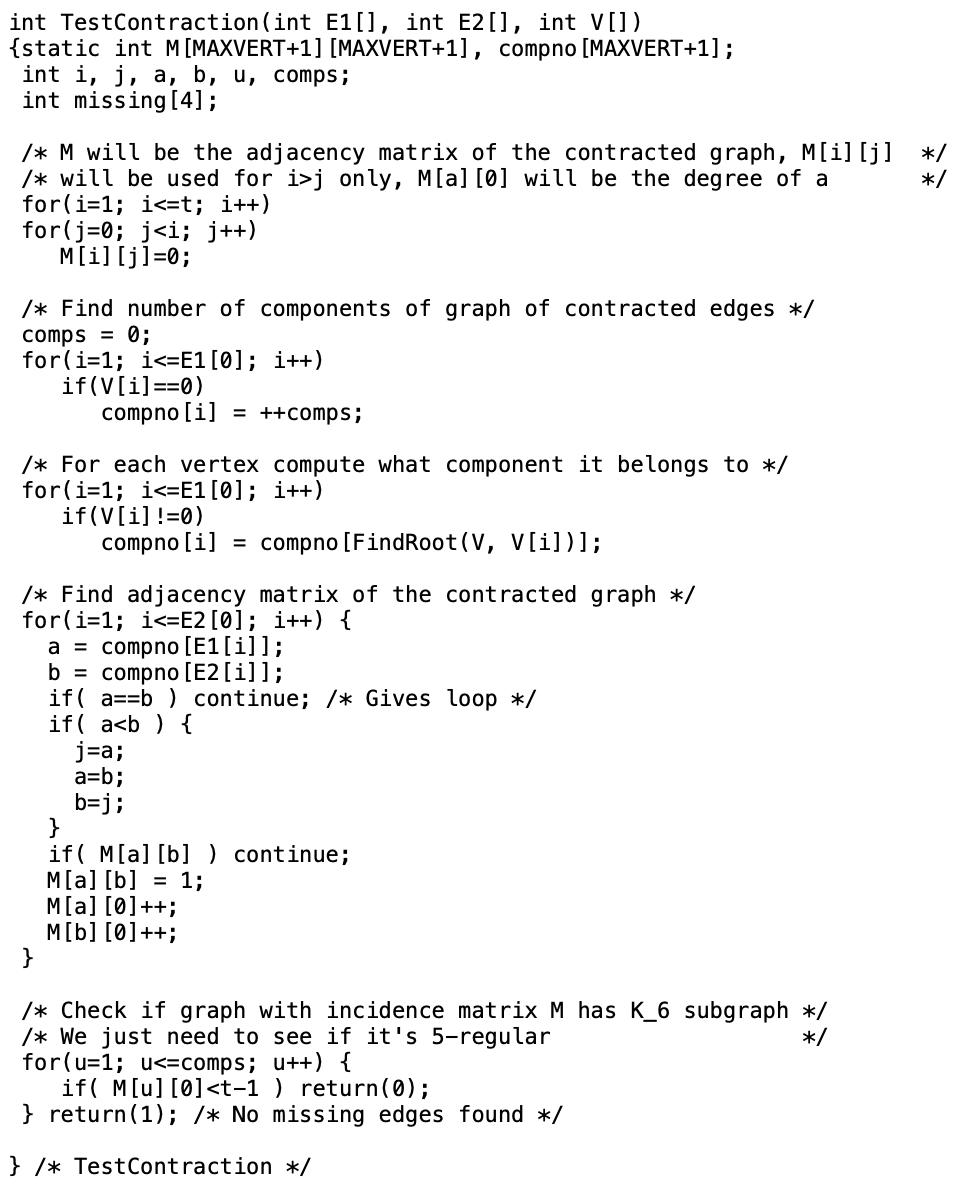}
 \end{figure}

\hskip 10cm

    \end{document}